\documentclass{amsart}
\usepackage[T1]{fontenc}
\usepackage{hyperref}

\usepackage{tikz}
\makeatletter

\newfont{\gothic}{eufm10 scaled 1100}

\theoremstyle{plain}    
\newtheorem{thm}{Theorem}[section]
\numberwithin{equation}{section} %% Comment out for sequentially-numbered
\numberwithin{figure}{section} %% Comment out for sequentially-numbered
\theoremstyle{plain}    
\newtheorem{cor}[thm]{Corollary} %%Delete [thm] to re-start numbering
\theoremstyle{plain}    
\newtheorem{conj}[thm]{Conjecture} %%Delete [thm] to re-start numbering
\theoremstyle{plain}    
\theoremstyle{plain}
\newtheorem{lem}[thm]{Lemma} %%Delete [thm] to re-start numbering
\theoremstyle{plain}    
\newtheorem{prop}[thm]{Proposition} %%Delete [thm] to re-start numbering
\theoremstyle{plain}    
\newtheorem{Def}[thm]{Definition} %%Delete [thm] to re-start numbering
\theoremstyle{remark}
\newtheorem{rem}[thm]{Remark}
\theoremstyle{remark}

\usepackage[arrow,matrix,curve,ps]{xy}
\usepackage{mathtools}
\usepackage{extarrows}

\makeatother

\begin{document}

\title{Numerical Analogues of the Kodaira Dimension and the Abundance Conjecture}

%\date{\today}

\author{Thomas Eckl}



%\thanks{}

\address{Thomas Eckl, Department of Mathematical Sciences, The University of Liverpool, Mathematical
               Sciences Building, Liverpool, L69 7ZL, England, U.K.}

\email{thomas.eckl@liv.ac.uk}

\urladdr{http://pcwww.liv.ac.uk/~eckl/}

\maketitle

\begin{abstract}
We add further notions to Lehmann's list of numerical analogues to the Kodaira dimension of pseudo-effective divisors on smooth complex projective varieties, and show new relations between them. Then we use these notions and relations to fill in a gap in Lehmann's arguments, thus proving that most of these notions are equal. Finally, we show that the Abundance Conjecture, as formulated in the context of the Minimal Model Program, and the Generalized Abundance Conjecture using these numerical analogues to the Kodaira dimension, are equivalent for non-uniruled complex projective varieties.
\end{abstract}

%\tableofcontents

\pagestyle{myheadings}
\markboth{\scshape{Thomas Eckl}}{\scshape{Numerical Analogues of the Kodaira Dimension}}

\setcounter{section}{-1}

\section{Introduction}

\noindent During the last decade a plethora of numerical analogues to the Kodaira dimension for pseudoeffective divisors on (smooth) complex projective varieties was introduced, by Nakayama \cite{Nak04}, Demailly, Boucksom, P\u{a}un and Peternell \cite{BDPP13}, Siu \cite{Siu11} and Lehmann \cite{Leh13}. Lehmann furthermore clarified lots of relations between these numerical dimensions, adding some new notions, ordering them by the way how they are constructed and showing that most of them are at least related by an inequality. However, his results contain a gap leaving the equality of most of these notions unproven, see the discussion in Section 2.9. This note fills in the gap in Chapter~3 extending results on the derivatives of the volume function in  \cite{BDPP13} and \cite{BFJ09}. Furthermore, we slightly extend Lehmann's list and prove some more relations. Finally we show that the Abundance Conjecture as formulated in the context of the Minimal Model Program (see e.g. \cite[Conj.3-3-4]{Mat02}) is equivalent to a Generalised Abundance Conjecture introduced in \cite{BDPP13}. On the way, we prove the birational equivalence of most of these notions of numerical dimension.  

\noindent In more detail, we will discuss the following notions of numerical dimension, ordered according to their construction method as suggested by Lehmann, and postponing some technical definitions to section~\ref{ND-defs-sec}:
\begin{Def} \label{ND-defs}
Let $X$ be a smooth complex projective variety and $D$ a pseudoeffective $\mathbb{R}$-divisor on $X$. Then we define the following numerical dimensions using
\begin{itemize}
\item positive product conditions:
\begin{enumerate}
\item $\nu_{\mathrm{K\ddot{a}h}}(D) := \max \left\{ k \in \mathbb{N} | \langle [D]^k \rangle_{\mathrm{K\ddot{a}h}} \neq 0 \right\}$, where $[D]$ denotes the $(1,1)$-cohomology class of the integration current associated to $D$;
\item $\nu_{\mathrm{alg}}(D) := \max \left\{ k \in \mathbb{N} | \langle D^k \rangle_{\mathrm{alg}} \neq 0 \right\}$;
\item $\nu_{\mathrm{res}}(D) := \max \left\{ \dim W | \langle D^{\dim W} \rangle_{X|W} > 0 \right\}$ where $ W \subset X$ ranges over subvarieties not contained in the diminished base locus $\mathbb{B}_-(D)$ (defined in \ref{nu-res-ssec});
\end{enumerate}
\item volume conditions:
\begin{enumerate}
\setcounter{enumi}{3}
\item $\nu_{\mathrm{Vol}}(D) := \max \left\{ k \in \mathbb{N} | \exists C > 0 : C \cdot t^{n-k} < \mathrm{vol}(D + tA)\ \mathrm{for\ all\ } t > 0 \right\}$, where $A$ is a sufficiently ample $\mathbb{Z}$-divisor on $X$;
\item $\nu_{\mathrm{Vol, res}}(D) := \max \left\{ \dim W | \lim_{\epsilon \rightarrow 0} \mathrm{vol}_{X|W}(D + \epsilon A) > 0 \right\}$, where $ W \subset X$ ranges over subvarieties not contained in $\mathbb{B}_-(D)$ and $A$ is a sufficiently ample $\mathbb{Z}$-divisor on $X$;
\item $\nu_{\mathrm{Vol, Zar}} := \max \left\{ \dim W | \inf_\phi \mathrm{vol}_{\widetilde{W}}(P_\sigma(\phi^\ast D)_{|\widetilde{W}}) > 0 \right\}$, where $ W \subset X$ ranges over subvarieties not contained in $\mathbb{B}_-(D)$, the morphism $\phi: (\widetilde{X}, \widetilde{W}) \rightarrow (X,W)$ ranges over all smooth $W$-birational models of $(X,W)$ and $P_\sigma(\phi^\ast D)$ is the positive part of the Zariski decomposition of $\phi^\ast D$ (all defined in \ref{nu-vol-zar-ssec});
\end{enumerate}
\item perturbed growth conditions:
\begin{enumerate}
\setcounter{enumi}{6}
\item $\kappa_{\sigma}(D) := \max \{ k \in \mathbb{N} | \limsup_{m \rightarrow \infty} m^{-k} h^0(X, \mathcal{O}_X(A+\lfloor mD \rfloor)) > 0 \}$, where $A$ is a sufficiently ample $\mathbb{Z}$-divisor;
\item $\kappa_{\mathrm{num}}(D) := \sup_{k \geq 1} \left\{ \limsup_{m \rightarrow \infty} \frac{\log h^0(X, \mathcal{O}_X(\lfloor mD \rfloor + kA))}{\log m} \right\}$, where $A$ is an ample $\mathbb{Z}$-divisor;
\end{enumerate}
\item Seshadri-type conditions:
\begin{enumerate}
\setcounter{enumi}{8}
\item $\kappa_\nu(D) := \min \left\{ \dim W | D \not\succeq W \right\}$, where $D \succeq W$ means that $D$ dominates $W$ (defined in \ref{kappa-nu-ssec});
\item $\kappa_{\nu, \mathrm{Leh}}(D) := \min \left\{ \dim W | \forall \epsilon > 0: \phi_W^\ast D - \epsilon E_W\ \mathrm{not\ pseudoeffective} \right\}$, where $\phi_W: \widetilde{X} \rightarrow X$ is any birational morphism of smooth varieties such that $\mathcal{O}_{\widetilde{X}}(E_W) = \phi^{-1}\mathcal{I}_W \cdot \mathcal{O}_{\widetilde{X}}$.
\end{enumerate}
\end{itemize}
\end{Def}

\noindent For attributions of these definitions see also section~\ref{ND-defs-sec}.

\begin{thm} \label{numdim-eq-thm}
Let $X$ be a smooth complex projective variety and $D$ a pseudoeffective $\mathbb{R}$-divisor on $X$. Then all the notions of numerical dimension listed in Def.~\ref{ND-defs} are equal, except $\kappa_{\nu, \mathrm{Leh}}(D)$ which may be smaller.
\end{thm}

\noindent This theorem is a consequence of the following net of equalities and inequalities:
\[ \arraycolsep=1.4pt
    \begin{array}{ccccccc}
    \nu_{\mathrm{Vol}}(D) &  &  & \kappa_{\mathrm{num}}(D) & \xlongequal{\rule{0.5cm}{0cm}} & \kappa_{\mathrm{num}}(D) & \\
    \rotatebox{90}{=} &  & &  \rotatebox{90}{$\leq$} & &  \rotatebox{90}{$\geq$} &  \\
    \nu_{\mathrm{alg}}(D) = & \nu_{\mathrm{res}}(D) = & \nu_{\mathrm{Vol,res}}(D) = & \nu_{\mathrm{Vol,Zar}}(D) \leq &  
    \kappa_\sigma(D) \leq & \kappa_\nu(D) \leq & \nu_{\mathrm{alg}}(D) \\
    \rotatebox{90}{$=$} & & & & &  \rotatebox{90}{$\leq$} & \\
    \nu_{\mathrm{K\ddot{a}h}}(D) & & & & & \kappa_{\nu, \mathrm{Leh}}(D) & 
    \end{array} \]

\noindent In section~\ref{ineq-ND-sec} and \ref{kn-nalg-sec} we will prove the equality $\nu_{\mathrm{alg}}(D) = \nu_{\mathrm{K\ddot{a}h}}(D)$ and the inequalities $\kappa_\nu(D) \leq \nu_{\mathrm{alg}}(D)$, $\nu_{\mathrm{Vol, Zar}}(D) \leq \kappa_{\mathrm{num}}(D)$, $\kappa_{\mathrm{num}}(D) \leq \kappa_\nu(D)$ and $\kappa_{\nu, \mathrm{Leh}}(D) \leq \kappa_\nu(D)$, and we locate the proofs of the other inequalities in the works of Lehmann~\cite{Leh13} and Nakayama~\cite{Nak04}. Our proofs of $\nu_{\mathrm{Vol, Zar}}(D) \leq \kappa_{\mathrm{num}}(D)$ and $\kappa_{\mathrm{num}}(D) \leq \kappa_\nu(D)$ also rely on Lehmann's ideas. 

\noindent But for the proof of $\kappa_\nu(D) \leq \nu_{\mathrm{alg}}(D)$ we need a new ingredient: the derivative of the restricted volume, generalizing Thm.A in \cite{BFJ09} (see Thm~\ref{der-resvol-thm}). Note that Lehmann \cite[Thm.5.3]{Leh13} already shows $\kappa_{\nu,\mathrm{Leh}}(D) \leq \nu_{\mathrm{alg}}(D)$ but his claim of equality $\kappa_{\nu,\mathrm{Leh}}(D) = \kappa_\nu(D)$ fails. Since the proof of $\kappa_\sigma(D) \leq \kappa_\nu(D)$ only works for Nakayama's definition of $\kappa_\nu(D)$, we need the new ingredient to close the gap.

\noindent The theorem shows that most of the notions in Def.~\ref{ND-defs} are equal. Therefore the following definition is justified:
\begin{Def} \label{nd-def}
Let $X$ be a smooth complex projective variety and $D$ a pseudoeffective $\mathbb{R}$-divisor on $X$. Then the numerical dimension $\nu_X(D)$ of $D$ is defined as one of the equal numbers
\[ \nu_{\mathrm{alg}}(D) = \nu_{\mathrm{res}}(D) = \nu_{\mathrm{Vol}}(D) = \nu_{\mathrm{Vol,res}}(D) =  \nu_{\mathrm{Vol,Zar}}(D) = \kappa_\sigma(D) = \kappa_\nu(D) = \kappa_{\mathrm{num}}(D), \]
and $\nu_X(D)$ only depends on the numerical class of $D$.
\end{Def}

\noindent In section~\ref{Bir-Abund-sec} we show that the numerical dimension of a pseudoeffective divisor behaves well under birational morphisms, following the ideas of Nakayama but explicitly using Thm.~\ref{numdim-eq-thm}:
\begin{prop}[\textbf{$=$ Proposition~\ref{bir-inv-nd-prop}}]
Let $f: \widetilde{X} \rightarrow X$ be a birational morphism between smooth complex projective varieties, let $D$ be a pseudoeffective divisor on $X$ and $\widetilde{D}$ a pseudoeffective divisor on $\widetilde{X}$ such that $\widetilde{D} - f^\ast D$ is an $f$-exceptional divisor. Then:
\[ \nu_X(D) = \nu_{\widetilde{X}}(\widetilde{D}). \]
\end{prop}

\noindent In a celebrated theorem Boucksom, Demailly, P\u{a}un and Peternell show that the canonical divisor $K_X$ of a non-uniruled smooth complex variety $X$ is pseudoeffective \cite[Cor.0.3]{BDPP13}. Consequently, the numerical dimension of the canonical divisor can be used to state the Abundance Conjecture:
\begin{conj}[Abundance Conjecture] \label{Abund-conj}
Let $X$ be a non-uniruled smooth complex projective variety. Then:
\[ \nu(X) := \nu_X(K_X) = \kappa(X). \]
\end{conj}

\noindent Here $\kappa(X) = \kappa_X(K_X)$ denotes the Kodaira dimension of the canonical divisor $K_X$, defined e.g. as
\[ \kappa_X(K_X) := \limsup_{m \rightarrow \infty} \frac{\log h^0(X, \mathcal{O}_X(mK_X))}{\log m}. \]
Note that  Boucksom, Demailly, P\u{a}un and Peternell refer in their Generalized Abundance Conjecture \cite[Conj.3.8]{BDPP13} to $\nu_{\mathrm{K\ddot{a}h}}(K_X)$ which is only conjecturally equal to $\nu(X)$, as discussed in \ref{nu-alg-nu-Kaeh-ineq-ssec}.

\noindent  In the context of the Minimal Model Program the Abundance Conjecture is formulated under the assumption that minimal models of smooth complex projective varieties exist (see Section~\ref{Bir-Abund-sec} for definitions):
\begin{conj}[Abundance Conjecture, MMP version {\cite[Conj.3-3-4]{Mat02}}] \label{MMP-Abund-conj}
Let $S$ be a minimal model of a non-uniruled smooth projective complex variety $X$. Then $|mK_S|$ is base point free for sufficiently divisible and large $m \in \mathbb{N}$ (that is, $K_S$ is semi-ample).
\end{conj}

\noindent We use the birational invariance of the numerical dimension to show in section~\ref{Bir-Abund-sec} that the two Abundance Conjectures as stated above are equivalent:
\begin{thm}[\textbf{$=$ Theorem~\ref{Abund-equiv-thm}}]
Let $S$ be a minimal model of a non-uniruled smooth projective complex variety $X$. Then
\[ \nu_X(K_X) = \kappa_X(K_X) \Longleftrightarrow K_S\ \mathrm{is\ semi-ample}. \]
\end{thm}

\noindent Note that this equivalence is asserted in passing in \cite{DHP13}) and proven in all detail in \cite[Thm.4.3]{GL13}, using results of \cite{Lai11}. Relying on \cite{BCHMcK10}, Gongyo and Lehmann were even able to show that $\nu_X(K_X) = \kappa_X(K_X)$ already implies the existence of a minimal model. 

\noindent However, the author still thinks that it is worth presenting the argument for Thm.~\ref{Abund-equiv-thm}, emphasizing in particular that not all the possible definitions of numerical dimension are easily shown to be birationally invariant.

\vspace{0.1cm}

\noindent \textit{Acknowledgements.} The author would like to thank the anonymous referees for finding and closing a gap in the proof of $\kappa_\nu(D) \leq \nu_{\mathrm{alg}}(D)$, and for the suggestions to improve the presentation of the paper, in particular for clarifying the newest developments on the equivalent formulations of the Abundance Conjecture.

\section{Notions of numerical dimension} \label{ND-defs-sec}

\noindent In the following $X$ is always a smooth $n$-dimensional complex projective variety and $D$ a pseudoeffective $\mathbb{R}$-divisor on $X$. 

\subsection{$\nu_{\mathrm{K\ddot{a}h}}(D)$}
This notion is defined in \cite[Def.3.6]{BDPP13}. The \textit{moving intersection product} $\langle [D]^k \rangle_{\mathrm{K\ddot{a}h}}$ of the $(1,1)$-cohomology class of the integration current $[D]$ is constructed in \cite[Thm.3.5]{BDPP13} following \cite{Bou02}: For suitably chosen birational morphisms $\mu_m: X_m \rightarrow X$ of smooth complex varieties, real numbers $\delta_m \downarrow 0$, closed semi-positive forms $\beta_{i,m}$ representing big and nef classes in $N^1(X_m)$, effective $\mu_m$-exceptional $\mathbb{Q}$-divisors $E_{i,m}$ on $X_m$ and any ample class $\omega$ on $X$ such that $[E_{i,m}] + \beta_{i,m}$ represents the $(1,1)$-class $\left( \mu_m \right)^\ast \left( [D] + \delta_m \omega \right)$ we can set
\[ \langle [D]^k \rangle_{\mathrm{K\ddot{a}h}} := \lim_{m \rightarrow \infty} \left( \mu_m \right)_\ast \left( \left[ \beta_{1,m} \wedge \ldots \wedge \beta_{k,m} \right] \right), \] 
where the limits are taken in $H^{k,k}(X)$. Note that other choices of $\mu_m, \delta_m, \beta_{i,m}, E_{i,m}$ satisfying the properties above will yield "smaller" $(k,k)$-classes $\alpha$, that is, $\langle [D]^k \rangle_{\mathrm{K\ddot{a}h}} - \alpha$ is represented by a positive current.

\subsection{$\nu_{\mathrm{alg}}(D)$} \label{nu-alg-ssec}
This notion appears first in \cite{Leh13} where it is nevertheless attributed to \cite{BDPP13}. In fact, Lehmann uses the algebraic analogue of the moving intersection product $\langle [D]^k \rangle_{\mathrm{K\ddot{a}h}}$ as defined in \cite{BFJ09}: To calculate $\langle [D]^k \rangle_{\mathrm{alg}}$ Boucksom, Favre and Jonsson replace the $(k,k)$-cohomology class $\left[ \beta_{1,m} \wedge \ldots \wedge \beta_{k,m} \right] \in H^{k,k}(X_m)$ by the intersection $k$-cycle class $[\beta_{1,m}] \cdots [\beta_{k,m}] \in N^k(X_m)$ and take the limit in $N^k(X)$. 

\noindent In particular, $\langle [D]^k \rangle_{\mathrm{alg}}$ only depends on the numerical class of $D$. The connection to $\langle [D]^k \rangle_{\mathrm{K\ddot{a}h}}$ is discussed in~\ref{nu-alg-nu-Kaeh-ineq-ssec}.

\noindent Note that the moving intersection product is continuous and homogeneous on the cone spanned by the classes of big divisors (\cite[Prop.2.9]{BFJ09}). Furthermore it coincides with the usual intersection number if the numerical classes are represented by nef divisors \cite[Prop.2.12]{BFJ09}.

\subsection{$\nu_{\mathrm{res}}(D)$} \label{nu-res-ssec}
This notion is defined in \cite{Leh13}. The \textit{diminished} or \textit{restricted base locus} of an $\mathbb{R}$-divisor 
\[ \mathbb{B}_{-}(D) := \bigcup_{\tiny \begin{array}{c} A\ \mathrm{ample}\\ D+A\ \mathbb{Q}-\mathrm{divisor} \end{array}} \mathbb{B}(D+A) \]
appears in \cite[Def.1.12]{ELM+06} and \cite[Def.III.2.6{\&}p.168]{Nak04}. Here,
\[ \mathbb{B}(D+A) := \bigcap_{m \geq 1} \mathrm{Bs}(\lfloor m(D+A) \rfloor) \] 
is the \textit{stable base locus} of the $\mathbb{Q}$-divisor $D+A$. Later on, we also need the \textit{augmented base locus}
\[ \mathbb{B}_+(B) := \bigcap_{\tiny \begin{array}{c} A\ \mathrm{ample}\\ B - A\ \mathbb{Q}-\mathrm{divisor} \end{array}} \mathbb{B}(B - A) \]
of a big $\mathbb{R}$-divisor $B$ (see~\cite[Def.1.2]{ELM+06}). Note that both base loci only depend on the numerical class of $D$ resp.\ $B$  (see~\cite[Prop.1.15]{ELM+06}).

\noindent The \textit{restricted moving intersection} $\langle D^k \rangle_{X|W}$ is constructed in \cite{BFJ09} for divisors $W$ and generalized to arbitrary subvarieties $W \not\subset \mathbb{B}_{-}(D)$ of dimension at least $k$ in \cite[Def.4.8]{Leh13} (then $D$ is called a \textit{$W$-pseudoeffective divisor}): Similar to \ref{nu-alg-ssec},
\[ \langle D^k \rangle_{X|W} := \lim_{m \rightarrow \infty} \left( \mu_m \right)_\ast ([B_{1,m}] \cdots [B_{k,m}] \cdot \widetilde{W}), \]
where the $B_{i,m}$ are suitably chosen big and nef divisors on the smooth variety $X_m$ such that $\mu_m: X_m \rightarrow X$ is a birational morphism whose center does not contain $W$ (a so-called $W$-\textit{birational model} of $X$), the $\mathbb{Q}$-divisors $\mu_m^\ast(D+\delta_mA) - B_{i,m}$ are effective and $\mu_m$-exceptional for a fixed ample divisor $A$ on $X$ and real numbers $\delta_m \downarrow 0$, and $\widetilde{W}$ is the strict $\mu_m$-transform of $W$. 

\noindent $\langle D^k \rangle_{X|W}$ only depends on the numerical class of $D$. On the cone spanned by classes of big divisors $B$ such that $W \not\subset \mathbb{B}_+(B)$ (then $B$ is called a \textit{$W$-big divisor}), the restricted product is continuous and homogeneous (see \cite[Prop.2.9\&Prop.4.6]{BFJ09} resp.\ \cite[Prop.4.7]{Leh13}). This implies furthermore that
\[ \langle D^k \rangle_{X|W} = \lim_{\delta \downarrow 0} \langle (D + B^{(\delta)})^k \rangle_{X|W}, \]
for arbitrary $W$-big divisors $B^{(\delta)}$ converging to $0$ when $\delta \downarrow 0$.

\noindent Note also that by setting $W := X$ the moving intersection cycle class $\langle D^k \rangle_{\mathrm{alg}}$ can be obtained as a special case of the restricted moving intersection product. Finally, in the calculation of $\langle D^k \rangle_{X|W}$ one can choose $B_{1,m} = \cdots = B_{k,m}$ (see the proof of \cite[Lem.2.6]{BFJ09}). 

\subsection{$\nu_{\mathrm{Vol}}(D)$} 
This notion is defined in \cite{Leh13}. Note that the volume of the big $\mathbb{R}$-divisor $D+tA$ can be defined as
\[ \mathrm{vol}(D+tA) := \limsup_{m \rightarrow \infty} \frac{h^0(X, \mathcal{O}_X(\lfloor m(D+tA)\rfloor))}{m^n} \]
because this definition coincides with the one in \cite[\S 2.2.C]{LazPAG1} as the continuous extension of the volume function on $\mathbb{Q}$-divisors to the big cone.

\noindent Fujita's theorem \cite[Thm.3.1]{BFJ09} states that $\mathrm{vol}(D+tA) = \langle (D+tA)^n \rangle_{\mathrm{alg}}$.

\subsection{$\nu_{\mathrm{Vol, res}}(D)$} \label{Vol-res-ssec}
This notion is introduced in \cite{Leh13} and uses the restricted volume investigated in \cite{ELM+09} (see also \cite[Def.2.12]{Leh13} for the definition):
\[ \mathrm{vol}_{X|W}(D+\epsilon A) := \limsup_{m \rightarrow \infty} \frac{h^0(X|W, \mathcal{O}_X(\lfloor m(D+\epsilon A) \rfloor))}{m^{\dim W}/(\dim W)!}, \]
where $H^0(X|W, \mathcal{O}_X(\lfloor m(D+\epsilon A) \rfloor))$ is defined as 
\[ \mathrm{Im}(H^0(X, \mathcal{O}_X(\lfloor m(D+\epsilon A) \rfloor)) \rightarrow H^0(W, \mathcal{O}_W(\lfloor m(D+\epsilon A) \rfloor))). \]
By the Generalised Fujita Theorem \cite[Prop.2.11\&Thm.2.13]{ELM+09},
\[ \langle (D+\epsilon A)^{\dim W} \rangle_{X|W} = \mathrm{vol}_{X|W}(D+\epsilon A). \]
Consequently, the restricted volume only depends on the numerical class of $D$ and is continuous and homogeneous on the cone spanned by the classes of $W$-big divisors~$B$.

\subsection{$\nu_{\mathrm{Vol, Zar}}(D)$} \label{nu-vol-zar-ssec}
Again this notion is introduced in \cite{Leh13}. Note that morphisms $\phi: (\widetilde{X}, \widetilde{W}) \rightarrow (X,W)$ are $W$-birational if the irreducible subvariety $W \subset X$ is not contained in the center of the birational map $\phi$, and $\widetilde{W}$ is the strict $\phi$-transform of $W$. The \textit{divisorial Zariski decomposition} or \textit{$\sigma$-decomposition}
\[ \phi^\ast D = P_\sigma(\phi^\ast D) + N_\sigma(\phi^\ast D) \]
into a positive part $P_\sigma$ and a negative part $N_\sigma$ is constructed by Nakayama \cite[III.1]{Nak04} and \cite{Bou04}.
Lehmann extracted from Nakayama's results in \cite[III.1]{Nak04} that the negative part $N_\sigma(\phi^\ast D)$ is the divisorial part of the diminished base locus $\mathbb{B}_{-}(\phi^\ast D)$ \cite[Prop.3.3(3)]{Leh13}, whereas Nakayama \cite[Lem.III.1.14(1)]{Nak04} showed that the numerical class of $P_\sigma(\phi^\ast D)$ lies in the closure of the movable cone $\mathrm{Mov}(X)$ spanned by fixed-part free divisors.

\noindent For later purposes we need more details of Nakayama's construction of the \mbox{$\sigma$-decomposition}:
\begin{Def} \label{ZD-def}
Let $X$ be a smooth projective complex variety, $B$ a big $\mathbb{R}$-divsor and $\Gamma$ a prime divisor on $X$. We set
\[ \sigma_\Gamma(B) := \inf \{ \mathrm{mult}_\Gamma \Delta | \Delta \equiv B, \Delta \geq 0 \}. \]
If $D$ is a pseudoeffective $\mathbb{R}$-divisor and $A$ an ample divisor on $X$ we set
\[ \sigma_\Gamma(D) := \lim_{\epsilon \downarrow 0} \sigma_\Gamma(D + \epsilon A) \] 
and define
\[ N_\sigma(D) := \sum_\Gamma \sigma_\Gamma(D) \cdot \Gamma. \]
\end{Def}

\noindent The well-definedness of $\sigma_\Gamma(D)$ is shown in \cite[III.1.5]{Nak04}; by \cite[III.1.11]{Nak04} $N_\sigma(D)$ is a finite sum.

\subsection{$\kappa_\sigma(D)$}
This notion is defined in \cite[Def.V.2.5]{Nak04}.

\subsection{$\kappa_{\mathrm{alg}}(D)$}
This notion is defined in \cite{Siu11}.

\subsection{$\kappa_\nu(D)$} \label{kappa-nu-ssec}
This notion is defined in \cite[Def.V.2.20]{Nak04}, requiring the notion of \textit{numerical dominance}:
\begin{Def}[{\cite[Def.V.2.12{\&}V.2.16]{Nak04}}]
Let $D$ be an $\mathbb{R}$-divisor on a smooth projective variety $X$ and $W \subset X$ an irreducible subvariety. We say that $D$ dominates $W$ numerically and write $D \succeq W$ if there exists a birational morphism $\phi: \widetilde{X} \rightarrow X$ and an ample divisor $A$ on $\widetilde{X}$ such that $\phi^{-1} \mathcal{I}_W \cdot \mathcal{O}_{\widetilde{X}} = \mathcal{O}_{\widetilde{X}}(E_W)$ is the locally free sheaf of an effective divisor $E_W$ on $\widetilde{X}$, and for every real number $b > 0$ there exist real numbers $x>b, y>b$ such that
\[ x \cdot \phi^\ast D - y \cdot E_W + A \]
is pseudoeffective.
\end{Def}

\noindent Note that the condition above is satisfied for any birational morphism $\psi: Y \rightarrow X$ with $\psi^{-1} \mathcal{I}_W \cdot \mathcal{O}_Y = \mathcal{O}_Y(F_W)$ for an effective divisor $F_W$ and ample divisor $B$ once it is satisfied for $\phi$ and $A$.

\subsection{$\kappa_{\nu, \mathrm{Leh}}(D)$}
This notion is introduced in \cite{Leh13} using \cite[Def.5.1]{Leh13}. See the discussion of the inequality $\kappa_{\nu, \mathrm{Leh}}(D) \leq \kappa_{\nu}(D)$ in \ref{k_nu_Leh_k_nu-ineq-ssec} for why the two invariants may be different.

\section{Inequalities between notions of numerical dimension} \label{ineq-ND-sec}

\subsection{$\nu_{\mathrm{alg}}(D) = \nu_{\mathrm{K\ddot{a}h}}(D)$} \label{nu-alg-nu-Kaeh-ineq-ssec}
The inequality $\nu_{\mathrm{alg}}(D) \leq \nu_{\mathrm{K\ddot{a}h}}(D)$ holds because $k$-cycles are numerically equivalent if the corresponding integration currents are cohomologically equivalent. 

\noindent Vice versa, $(k,k)$-classes $\left( \mu_m \right)_\ast \left( \left[ \beta_{1,m} \wedge \ldots \wedge \beta_{k,m} \right] \right)$ whose limit calculates $\langle [D]^k \rangle_{\mathrm{K\ddot{a}h}}$, correspond by construction to numerical $k$-cycle classes $\left( \mu_m \right)_\ast \left( [\beta_{1,m}] \cdots [\beta_{k,m}]\right) \leq \langle [D]^k \rangle_{\mathrm{alg}}$. So the numerical class $\alpha$ corresponding to $\langle [D]^k \rangle_{\mathrm{K\ddot{a}h}}$ is $\leq \langle [D]^k \rangle_{\mathrm{alg}}$.

\noindent Since both numerical classes are pseudoeffective $\langle [D]^k \rangle_{\mathrm{alg}} = 0$ implies $\alpha = 0$. Hence $\alpha \cdot H^{n-k} = 0$ for an ample divisor $H$ on $X$, and for a positive $(k,k)$-current $T$ representing $\langle [D]^k \rangle_{\mathrm{K\ddot{a}h}}$ we have $\int_T \omega_H^{n-k} = 0$ where $\omega_H$ is the K\"ahler form associated to $H$. But this is only possible if $T = 0$, that is $\langle [D]^k \rangle_{\mathrm{K\ddot{a}h}} = 0$.
The inequality $\nu_{\mathrm{K\ddot{a}h}}(D) \leq \nu_{\mathrm{alg}}(D)$ follows.

\subsection{$\nu_{\mathrm{alg}}(D) \leq \nu_{\mathrm{Vol}}(D)$} 
This inequality is proven in \cite[Thm.6.2.(1){$=$}(7)]{Leh13}.

\subsection{$\nu_{\mathrm{alg}}(D) \leq \nu_{\mathrm{res}}(D)$}
This inequality  is proven in \cite[Thm.6.2.(1){$=$}(2)]{Leh13}.

\subsection{$\nu_{\mathrm{res}}(D) \leq \nu_{\mathrm{Vol, res}}(D)$}
This inequality  is proven in \cite[Thm.6.2.(2){$=$}(3)]{Leh13}.

\subsection{$\nu_{\mathrm{Vol, res}}(D) \leq \nu_{\mathrm{Vol, Zar}}(D)$}
This inequality  is proven in \cite[Thm.6.2.(3){$=$}(4)]{Leh13}.

\subsection{$\nu_{\mathrm{Vol, Zar}}(D) \leq \kappa_{\sigma}(D)$}
This inequality  is proven in \cite[Thm.6.2.(4){$\leq$}(5)]{Leh13}.

\subsection{$\kappa_{\sigma}(D) \leq \kappa_{\nu}(D)$}
This inequality is proven in \cite[Prop.V.2.22(1)]{Nak04}.

\subsection{$\nu_{\mathrm{Vol, Zar}}(D) \leq \kappa_{\mathrm{num}}(D)$}
For a sufficiently ample divisor $A$ Lehmann shows in \cite[Thm.6.2.(4){$\leq$}(5)]{Leh13} that there exists a constant $C > 0$ so that for every suffciently large $m$
\[ C m^{\nu_{\mathrm{Vol, Zar}}(D)} \leq h^0(X, \mathcal{O}_X(\lfloor mD \rfloor + A)). \]
Taking the logarithm, dividing by $\log m$ and letting $m$ tend to $\infty$ shows the desired inequality.

\subsection{$\kappa_{\nu, \mathrm{Leh}}(D) \leq \kappa_\nu(D)$} \label{k_nu_Leh_k_nu-ineq-ssec}
Let $W \subset X$ be an irreducible subvariety, $\phi: \widetilde{X} \rightarrow X$ a birational morphism of smooth varieties such that $\mathcal{O}_{\widetilde{X}}(E_W) = \phi^{-1} \mathcal{I}_W \cdot \mathcal{O}_{\widetilde{X}}$ and $A$ an ample divisor on $\widetilde{X}$. If $\phi^\ast D - \epsilon E_W$ is pseudoeffective for an $\epsilon > 0$ then 
\[ \frac{b+1}{\epsilon} \phi^\ast D - (b+1) E_W + A\] 
is also pseudoeffective, for any $b > 0$, hence $D \succeq W$. Consequently, 
\[ \kappa_{\nu, \mathrm{Leh}}(D) \leq \kappa_{\nu}(D). \]
Note that the argument for equality in the proof of \cite[Prop.5.3]{Leh13} does not work because projections of finite-dimensional vector spaces are not closed maps. In particular equality could fail if $\phi^\ast D$ sits on a non-polyhedral part of the boundary of the big cone $\mathrm{Big}(\widetilde{X})$, as illustrated in the following diagram of a cut through the big cone by the affine plane in $\mathrm{NS}(\widetilde{X})_{\mathbb{R}}$ passing through $E_W$, $\phi^\ast D$ and $E_W - \frac{1}{b+1}A$, for arbitrary $b > 0$:

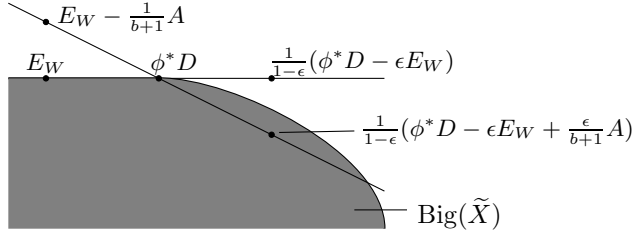
\vspace{0.2cm}

\begin{center}
\begin{tikzpicture}
   \draw (-2,0) -- (3,0);
   \draw[thick] (-2,0) -- (0,0);
   \draw[thick] (0,0) .. controls (1,0) and (3,-1) .. (3,-2);
   \fill[gray] (-2,-2)  -- (-2,0) -- (0,0) .. controls (1,0) and (3,-1) .. (3,-2) -- cycle; 
   \draw (2.6 ,-1.75) -- (3.3,-1.75);
   \node at (4,-1.8) {$\mathrm{Big}(\widetilde{X})$};

   \filldraw (-1.5,0) circle (1pt); \node at (-1.5,0.2) {\small $E_W$};
   \filldraw (-1.5, 0.75) circle (1pt); \node at (-0.5,0.8) {\small $E_W - \frac{1}{b+1}A$};
   \filldraw (0,0) circle (1pt); \node at (0.2, 0.2) {\small $\phi^\ast D$}; 
   \filldraw (1.5,0) circle (1pt); \node at (2.7,0.2) {\small $\frac{1}{1-\epsilon}(\phi^\ast D - \epsilon E_W)$};
   \filldraw (1.5, -0.75) circle (1pt); 
   \node at (4.5, -0.7) {\small $\frac{1}{1-\epsilon}(\phi^\ast D - \epsilon E_W + \frac{\epsilon}{b+1}A)$};
   \draw (2.6,-0.7) -- (1.6,-0.72);

   \draw (-2, 1) -- (3,-1.5);
\end{tikzpicture}
\end{center}

\vspace{0.2cm}

\noindent In this situation, $\phi^\ast D - \epsilon E_W$ is not pseudoeffective for all $\epsilon > 0$, but $\phi^\ast D - \epsilon E_W + \frac{\epsilon}{b+1}A$ is pseudoeffective for all sufficiently small $\epsilon > 0$. Consequently, $\frac{b+1}{\epsilon}\phi^\ast D -  (b+1)E_W + A$ is pseudoeffective, hence $D \succeq W$.

\noindent Note also that Nakayama's proof of $\kappa_\sigma(D) \leq \kappa_\nu(D)$ does not work if we replace $\kappa_\nu(D)$ with $\kappa_{\nu, \mathrm{Leh}}(D)$: The definition of $\kappa_\sigma(D)$ only allows one to find sections of $\mathcal{O}_X(\lfloor mD \rfloor + A)$, with the ample divisor $A$ on $X$ added.

\subsection{$\kappa_{\mathrm{num}}(D) \leq \kappa_\nu(D)$}
We adapt \cite[Thm.6.7(7)]{Leh13} and its proof to Nakayama's definition of $\kappa_\nu(D)$: Let $A$ be a sufficiently ample divisor on $X$, and let $W \subset X$ be a subvariety such that $\dim W = \kappa_\nu(D)$ and $D \not\succeq W$. In particular, for a resolution $\phi: \widetilde{X} \rightarrow X$ of $W$ and an ample divisor $H$ on the smooth projective variety $\widetilde{X}$, there exists $b > 0$ such that $x\phi^\ast D - y E_W + H$ is not pseudoeffective for any choice of $x, y > b$.

\noindent Choose $q \in \mathbb{N}$ large enough so that $qH - \phi^\ast A$ is pseudoeffective, and consider any sufficiently large $m \in \mathbb{N}$. Then the $\mathbb{R}$-divisor $m\phi^\ast D - q \lceil b+1 \rceil E_W + qH$ and hence 
\[ m\phi^\ast D - q \lceil b+1 \rceil E_W + qH - (qH - \phi^\ast A) = \phi^\ast (mD + A) - q \lceil b+1 \rceil E_W \]
is not pseudoeffective. Therefore $\phi^\ast (\lfloor mD \rfloor + A) - q \lceil b+1 \rceil E_W$ is not effective, and we obtain
\[ h^0(\widetilde{X}, \mathcal{O}_{\widetilde{X}}(\phi^\ast (\lfloor mD \rfloor + A) - q \lceil b+1 \rceil E_W)) = 0. \]
Consequently, $h^0(X, \mathcal{O}_X(\lfloor mD \rfloor + A) \otimes \mathcal{I}_W^{q \lceil b+1 \rceil}) = 0$. Set $q^\prime = q \lceil b+1 \rceil$ and write $W_{q^\prime}$ for the subscheme defined by the ideal sheaf $\mathcal{I}_W^{q^\prime}$. Then there is an injection
\[ H^0(X, \mathcal{O}_X(\lfloor mD \rfloor + A)) \rightarrow H^0(W_{q^\prime}, \mathcal{O}_{W_{q^\prime}}(\lfloor mD \rfloor + A)). \]
Since $\lfloor mD \rfloor + A \leq m\lfloor D + A\rfloor$ the rate of growth for the right hand side is bounded by a multiple of $m^{\dim W_{q^\prime}} = m^{\kappa_\nu(D)}$. In particular, there exists a constant $C > 0$ such that
\[ h^0(X, \mathcal{O}_X(\lfloor mD \rfloor + A)) \leq C \cdot m^{\kappa_\nu(D)}. \]
Taking the logarithm, dividing by $\log m$ and letting $m$ tend to $\infty$ shows the desired inequality.

\section{Proof of $\kappa_\nu(D) \leq \nu_{\mathrm{alg}}(D)$} \label{kn-nalg-sec}

\noindent To show this inequality we cannot just adapt the proof of \cite[Thm.6.2(6) {$\leq$} (1)]{Leh13} to Nakayama's definition of $\kappa_\nu(D)$ but need a new ingredient: the derivative of the restricted volume function. The following statement generalizes Thm.A in \cite{BFJ09}.
\begin{thm} \label{der-resvol-thm} 
Let $X$ be a $n$-dimensional smooth projective complex variety and $V = H_1 \cap \ldots \cap H_{n-k}$ a $k$-dimensional complete intersection variety cut out by very general very ample  linearly equivalent divisors $H_i$. If $\alpha$ is a $V$-big  and $\gamma$ an arbitrary divisor class then
\[ \frac{d}{dt}_{|t=0} \mathrm{vol}_{X|V}(\alpha + t\gamma) = k \cdot \langle \alpha^{k-1} \rangle_{X|V} \cdot \gamma. \]
\end{thm}

\noindent To prove this theorem and the inequality we first need further facts on the restricted moving intersection product and volume.
\begin{lem} \label{res-free-lem}
Let $X$ be a smooth projective complex variety, $V \subset X$ a subvariety and $D$ a $V$-pseudoeffective divisor on $X$. Furthermore, let $F \subset X$ be a very general element of a free family of subvarieties, that is, a general element of the family intersects any given algebraic subset of $X$ in the expected codimension. Then for $k \leq \dim V \cap F$:
\[ \langle D^k \rangle_{X|V} \cdot F = \langle D^k \rangle_{X|V \cap F}. \]
\end{lem}
\begin{proof}
This is a generalisation of \cite[Lem.4.18(2)]{Leh13}: Consider a countable set of smooth $V$-birational models $\phi_m: (\widetilde{X}_m, \widetilde{V}_m) \rightarrow (X,V)$ on which the restricted product can be calculated, as 
\[ \langle D^k \rangle_{X|V} = \lim_{m \rightarrow \infty} \left( \phi_m \right)_\ast \left( [B_{1,m}] \cdots [B_{k,m} ]\right) \]
for big and nef divisors $B_{i,m}$ on $\widetilde{X}_m$. Choose $F$ sufficiently general so that it does not contain any of the $\phi_m$-exceptional centers and intersects $V$ generically transversally. Then the strict transform $\widetilde{V \cap F}$ of $V \cap F$ on $\widetilde{X}_m$ will be a cycle representing the class $[\phi_m^\ast F] \cdot [\widetilde{V}_m]$. Thus we can identify the classes
\begin{eqnarray*}  
\left( \phi_m \right)_\ast \left( [B_{1,m}] \cdots [B_{k,m}] \cdot \widetilde{V}_m \right) \cdot F & =  & 
             \left( \phi_m \right)_\ast \left( [B_{1,m}] \cdots [B_{k,m} ] \cdot [\phi_m^\ast F] \cdot \widetilde{V}_m \right) = \\
 & =  & \left( \phi_m \right)_\ast \left( [B_{1,m}] \cdots [B_{k,m}] \cdot [\widetilde{V \cap F}] \right), 
\end{eqnarray*}
and that implies the claimed equality.
\end{proof}

\begin{lem} \label{neq-prod-pos-int-lem}
Let $X$ be a smooth projective complex variety, $V \subset X$ a subvariety of dimension $d$ and $D$ a $V$-pseudoeffective divisor on $X$. If $k \leq d$ and $A$ is an ample divisor on $X$ then
\[ \langle D^k \rangle_{X|V} \neq 0 \Longleftrightarrow \langle D^k \rangle_{X|V} \cdot A^{d-k} > 0. \]
\end{lem}
\begin{proof}
The Chern character isomorphism $K(X)_{\mathbb{Q}} \rightarrow A(X)_\mathbb{Q}$ shows that $\langle D^k \rangle_{X|V}$ is numerically trivial if and only if $\langle D^k \rangle_{X|V} \cdot \alpha = 0$ for all $(n-d+k)$-cycles $\alpha \in A_{n-d+k}(X)$ (see \cite[Ex.19.1.5]{FIS}). 

\noindent Since $ A_{n-d+k}(X)$ is generated by subvarieties of codimension $d-k$ the pseudo-effective class $\langle D^k \rangle_{X|V} \not\equiv 0$ if and only if there exists a subvariety $Y \subset X$ of codimension $d-k$ such that $\langle D^k \rangle_{X|V} \cdot Y > 0$. But $Y$ is a component of a complete intersection $A_1 \cap \ldots \cap A_{d-k}$ of hyperplane sections $A_i \in |lA|$ for some $l \gg 0$, that is $A_1 \cap \ldots \cap A_{d-k} = Y_1 \cdots Y_s + Y^\prime$ for some subscheme $Y^\prime \subset X$. So we have 
\[
l^{d-k} \cdot \langle D^k \rangle_{X|V} \cdot A^{d-k}  = \langle D^k \rangle_{X|V} \cdot Y + \langle D^k \rangle_{X|V} \cdot Y^\prime  \geq    \langle D^k \rangle_{X|V} \cdot Y  > 0. \]

\noindent The opposite direction is obvious.
\end{proof}

\begin{prop} \label{resvol-diff-prop}
Let $X$ be an $n$-dimensional  smooth projective complex variety, $V = H_1 \cap \ldots \cap H_{n-k} \subset X$ a $k$-dimensional complete intersection subvariety cut out by very general free big and nef divisors $H_i$ linearly equivalent to $H$ and $A, B$ $V$-big and nef $\mathbb{R}$-divisors. Then:
\[ \mathrm{vol}_{X|V}(A-B) \geq A^k \cdot H^{n-k} - k \cdot A^{k-1} \cdot B \cdot H^{n-k}. \]
\end{prop}
\begin{proof}
This is a generalisation of \cite[Thm.2.2.15]{LazPAG1}. By continuity of the usual intersection product it is enough to choose an ample divisor $H^\prime$ and prove the inequality for $A + \epsilon H^\prime, B + \epsilon H^\prime$, that is for ample $\mathbb{R}$-divisors $A, B$. Since the restricted volume is continuous and homogeneous on the cone spanned by the classes of ample divisors, we can even assume that $A, B$ are very ample divisors.

\noindent Let us fix $m > 0$ and choose $m$ general divisors $B_1, \ldots, B_m \in |B|$. Then we have a commutative diagram

{\tiny
\xymatrix{
0 \ar[r] & H^0(X, \mathcal{O}_X(mA - \sum_{i=1}^m B_i)) \ar[r] \ar[d] &  H^0(X, \mathcal{O}_X(mA)) \ar[r] \ar[d] &  
               \bigoplus_{i=1}^m H^0(B_i, \mathcal{O}_{B_i}(mA)) \ar[d] \\
0 \ar[r] &  H^0(X|V, \mathcal{O}_X(mA - \sum_{i=1}^m B_i)) \ar[r] \ar[d] &  H^0(X|V, \mathcal{O}_X(mA)) \ar[r] \ar[d] &
                \bigoplus_{i=1}^m H^0(B_i|V \cap B_i, \mathcal{O}_{B_i}(mA)) \ar[d]\\ 
0 \ar[r]&  H^0(V, \mathcal{O}_V(mA - \sum_{i=1}^m B_i)) \ar[r] &  H^0(V, \mathcal{O}_V(mA)) \ar[r] &     
                      \bigoplus_{i=1}^m H^0(V \cap B_i, \mathcal{O}_{V \cap B_i}(mA)) }
}

\noindent where in the upper row the vertical arrows correspond to surjective maps whereas in the lower row the vertical arrows correspond to inclusions. Consequently,
\[ h^0(X|V, \mathcal{O}_X(m(A-B))) \geq h^0(X|V, \mathcal{O}_X(mA)) - \sum_{i=1}^m h^0(B_i|V \cap B_i, \mathcal{O}_{B_i}(mA)). \]
Dividing by $\frac{m^k}{k!}$ and going to the limit $m \rightarrow \infty$ we obtain
\begin{eqnarray*}
\mathrm{vol}_{X|V}(A-B) & \geq & \mathrm{vol}_{X|V}(A) - \sum_{i=1}^m \frac{k}{m} \mathrm{vol}_{B_i|V \cap B_i}(A) \\
 & = & \langle A^k \rangle_{X|V} - \sum_{i=1}^m \frac{k}{m} \langle A^{k-1} \rangle_{B_i|V \cap B_i} \\
 & = & \langle A^k \rangle_X \cdot H^{n-k} - \sum_{i=1}^m \frac{k}{m} \langle A^{k-1} \rangle_{B_i} \cdot H^{n-k} \\
 & = & A^k \cdot H^{n-k} - k \cdot A^{k-1} \cdot B \cdot H^{n-k}
\end{eqnarray*}
using the Generalised Fujita Theorem (see \ref{Vol-res-ssec}), Lemma~\ref{res-free-lem} and the ampleness resp.\ freeness of $A$, $H$ and the $B_i$. 
\end{proof}

\noindent In the following, $D_1 \leq_V D_2$ means that the difference $D_2 - D_1$ of the two $\mathbb{R}$-divisors $D_1, D_2$ on $X$ is effective and the support of $D_2 - D_1$ does not contain the subvariey $V \subset X$.

\begin{prop} \label{resvol-expand-prop}
Let $X$ be an $n$-dimensional  smooth projective complex variety, $V = H_1 \cap \ldots \cap H_{n-k} \subset X$ a $k$-dimensional complete intersection subvariety cut out by very general free big and nef divisors $H_i$ linearly equivalent to $H$ and $B$ a big and nef $\mathbb{R}$-divisor such that $B \leq_V H$. If $\gamma$ is an arbitrary divisor class such that $H \pm \gamma$ is still nef then
\[ \mathrm{vol}_{X|V}(B+t\gamma) \geq B^k \cdot H^{n-k} + k \cdot t \cdot B^{k-1} \cdot \gamma \cdot H^{n-k} - c \cdot t^2 \]
for every $0 \leq t \leq 1$ and some constant $c>0$ only depending on $H^n$.
\end{prop}
\begin{proof}
This is a generalisation of \cite[Cor.3.4]{BFJ09}. As in \cite[Cor.2.4]{BFJ09} we can use the assumption that $H \pm \gamma$ is nef to conclude that for $0 \leq t \leq 1$ and some constant $c^\prime > 0$ only depending on $H^n$,
\[ (B + t\gamma)^k \cdot H^{n-k} \geq B^k \cdot H^{n-k}+ k \cdot t \cdot B^{k-1} \cdot \gamma \cdot H^{n-k} - c^\prime \cdot t^2, \]
by replacing $\gamma$ with $(H + \gamma) - H$ and using that $H + \gamma \leq 2H$. If we also write $B + t\gamma$ as the difference of the two nef classes $C := B + t(\gamma+H)$ and $D := tH$ then we have furthermore
\[ (B + t\gamma)^k \cdot H^{n-k} = (C-D)^k \cdot H^{n-k} \leq C^k \cdot H^{n-k} - k \cdot C^{k-1} \cdot D \cdot H^{n-k} + c^{\prime\prime} \cdot t^2, \]
where $c^{\prime\prime}$ once again only depends on $H^n$: Indeed, $c^{\prime\prime}$ is controlled by $C^i \cdot H^{n-i}$, 
$0 \leq i \leq k-2$, and we have $C \leq 3H$. Thus we have
\[ C^k \cdot H^{n-k} - k \cdot C^{k-1} \cdot D \cdot H^{n-k} \geq  B^k \cdot H^{n-k}+ k \cdot t \cdot B^{k-1} \cdot \gamma \cdot H^{n-k} - (c^\prime + c^{\prime\prime})t^2. \]
The result follows by applying Prop.~\ref{resvol-diff-prop} to $B + t\gamma = C - D$.
\end{proof}

\noindent \textit{Proof of Thm.~\ref{der-resvol-thm}.} Let $H$ be a very general divisor linearly equivalent to the $H_i$, and assume that $\alpha$ is represented by the $\mathbb{R}$-divisor $A \leq_V H$ and that $H \pm \gamma$ is nef. If this is not the case replace $\alpha, \gamma$ by multiples $s\alpha, s\gamma$ with $s > 0$ sufficiently small. The claim for $\alpha, \gamma$ still follows, by homogeneity of restricted volumes and moving intersection numbers. 

\noindent Let $\beta$ be a nef divisor class on a $V$-birational model $\phi: (\widetilde{X}, \widetilde{V}) \rightarrow (X,V)$ such that $\beta$ is represented by the $\mathbb{R}$-divisor $B \leq_{\widetilde{V}} \phi^\ast \alpha$, hence also $B \leq_{\widetilde{V}} \phi^\ast H$. Since $\widetilde{V}$ is cut out by the big and nef divisors $\phi^\ast H_i$ Prop.~\ref{resvol-expand-prop} shows
\[ \mathrm{vol}_{X|V}(\alpha + t\gamma) \geq \mathrm{vol}_{\widetilde{X}|\widetilde{V}}(\beta + t\phi^\ast\gamma) \geq \beta^k \cdot (\phi^\ast H)^{n-k} + k \cdot t \cdot \beta^{k-1} \cdot \phi^\ast \gamma \cdot (\phi^\ast H)^{n-k} - c \cdot t^2 \]
for every $0 \leq t \leq 1$ and some constant $c > 0$ only depending on $H^n$. Taking the supremum over all nef classes $\beta \leq_{\widetilde{V}} \phi^\ast \alpha$ yields
\[ \mathrm{vol}_{X|V}(\alpha + t\gamma) \geq \mathrm{vol}_{X|V}(\alpha) + k \cdot t \cdot \langle \alpha^{k-1} \rangle_{X|V} \cdot \gamma - c \cdot t^2, \]
using Lem.~\ref{res-free-lem} and the Generalised Fujita Theorem. This holds for every $0 \leq t \leq 1$, and in fact also for every $-1 \leq t \le 0$, by replacing $\gamma$ with $-\gamma$.

\noindent Exchanging the roles of $\alpha + t\gamma$ represented by an $\mathbb{R}$-divisor $A^\prime \leq 2H$ and $\alpha = (\alpha + t\gamma) + t \cdot (-\gamma)$ we obtain
\[ \mathrm{vol}_{X|V}(\alpha) \geq \mathrm{vol}_{X|V}(\alpha+t\gamma) - k \cdot t \cdot \langle (\alpha+t\gamma)^{k-1} \rangle_{X|V} \cdot \gamma - c^\prime \cdot t^2 \]
for a constant $c^\prime$ possibly larger than $c$ but still only depending on $H^n$. Combining the two inequalities shows that
\[ \frac{d}{dt}_{|t=0} \mathrm{vol}_{X|V}(\alpha+t\gamma) = k \cdot \langle \alpha^{k-1} \rangle_{X|V} \cdot \gamma \]
as desired, since $\langle (\alpha+t\gamma)^{k-1} \rangle_{X|V}$ converges to $\langle \alpha^{k-1} \rangle_{X|V}$ if $t \rightarrow 0$.
\hfill $\Box$

\noindent To prove $\kappa_\nu(D) \leq \nu_{\mathrm{alg}}(D)$ we finally need to connect divisorial Zariski decomposition and algebraic moving intersection product. For the K\"ahler intersection product this was done in \cite[Thm.3.5]{BDPP13}.
\begin{prop}[{\cite[Ex.4.10]{Leh13}}] \label{alg-ZD-prop}
Let $X$ be a smooth projective complex variety and $D$ a pseudoeffective $\mathbb{R}$-divisor. Then the negative part of the divisorial Zariski decomposition $D = P_\sigma(D) + N_\sigma(D)$ can be calculated as
\[ N_\sigma(D) = D - \langle [D] \rangle_{\mathrm{alg}}. \]
\end{prop}

\noindent \textit{Proof of $\kappa_\nu(D) \leq \nu_{\mathrm{alg}}(D)$.} 
First assume that $\nu_{\mathrm{alg}}(D) = 0$. By definition this means that the positive product $\langle D \rangle_{\mathrm{alg}} = 0$, hence $P_\sigma(D) \equiv 0$ and $D \equiv N_\sigma(D)$ by Prop.~\ref{alg-ZD-prop}. Consequently, $\kappa_\nu(D) = 0$ by \cite[V.2.22(2)]{Nak04}.

\noindent So from now on we assume $1 \leq k := \nu_{\mathrm{alg}}(D) < \kappa_\nu(D) \leq n := \dim X$ and derive a contradiction: Fix an ample divisor $H$ on $X$ and $\epsilon > 0$. Choose birational models $\pi_i: X_i \rightarrow X$ and big and nef divisors 
\[ B_i \leq \pi_i^\ast (D + \epsilon H) \]
on $X_i$ such that the limit of the push forwards $\pi_{i\ast}(B_i^k)$ calculates the product $\langle (D+\epsilon H)^k \rangle_{\mathrm{alg}}$. \cite[Prop.3.5\&3.7]{Leh13} tell us that for a suitable effective divisor $G$ on $X$ we can further assume that the big and nef divisors $B_i$ satisfy
\[ B_i \leq P_\sigma(\pi_i^\ast (D + \epsilon H)) \leq B_i + \frac{1}{i} \pi_i^\ast G. \]

\noindent Let $W$ be a smooth $k$-dimensional intersection of very general very ample divisors such that $W \not\subset \mathbb{B}_-(D)$ and for all $i$, the centers of $\pi_i$ intersect $W$ transversally. Set $\phi: Y \rightarrow X$ to be the blow up of $X$ along $W$, with exceptional divisor $E$. Fix a very ample divisor $H_Y$ on $Y$. By \cite[V.2.21]{Nak04} $k < \kappa_\nu(D)$ implies that $D \succeq W$, that is, for each sufficiently small $\epsilon >0$ there exists a $\tau > 0$ such that $\phi^\ast D - \tau E + \epsilon H_Y$ is pseudoeffective and $\frac{\tau}{\epsilon} \rightarrow \infty$ when $\epsilon \rightarrow 0$.

\noindent Choose smooth birational models $\psi_i: (Y_i, E_i) \rightarrow (Y, E)$ lifting the birational maps $\pi_i: X_i \rightarrow X$ 
on which the restricted product $\langle (\phi^\ast (D + \epsilon H) + \epsilon H_Y)^k \rangle_{Y|E}$ can be computed. In particular there exist birational maps $\phi_i: Y_i \rightarrow X_i$ such that $\pi_i \circ \phi_i = \phi \circ \psi_i$ and big and nef divisors $A_i \leq_{E_i} \psi_i^\ast(\phi^\ast (D + \epsilon H) + \epsilon H_Y)$ on $Y_i$ such that the limit of the push forwards $\psi_{i\ast}(A_i^k \cdot E_i)$ calculates the restricted product. By the choice of $W$, $\phi_i(E_i)$ is the strict transform of $W$ on $X_i$. Hence $\phi_i^\ast B_i + \epsilon \psi_i^\ast H_Y \leq_{E_i} \psi_i^\ast(\phi^\ast (D + \epsilon H) + \epsilon H_Y)$, and  as in the proof of \cite[Lem.2.6]{BFJ09} we can achieve that $A_i \geq_{E_i} \phi_i^\ast B_i + \epsilon \psi_i^\ast H_Y$, by possibly further blowing up $Y_i$.

\noindent Since $W \not\subset \mathbb{B}_-(D)$, $E$ is not a component of $N_\sigma(\phi^\ast(D + \epsilon H))$. By \cite[III.5.16]{Nak04} $N_\sigma(\psi_i^\ast \phi^\ast(D + \epsilon H)) \geq \psi_i^\ast N_\sigma(\phi^\ast(D + \epsilon H))$, and every component in the difference is $\psi_i$-exceptional. Consequently, $E_i$ is not a component of $N_\sigma(\psi_i^\ast \phi^\ast(D + \epsilon H))$ either. Furthermore,  $\psi_i^\ast \phi^\ast (D + \epsilon H) + \epsilon \psi_i^\ast H_Y - \tau E_i \geq \psi_i^\ast \phi^\ast D + \epsilon \psi_i^\ast H_Y - \tau \psi_i^\ast E$ is pseudoeffective.

\noindent Then we can deduce that $P_\sigma(\psi_i^\ast \phi^\ast (D + \epsilon H)) + \epsilon \psi_i^\ast H_Y - \tau E_i$ is pseudoeffective: There is a pseudo-effective divisor $F$ such that 
\[ P_\sigma(\psi_i^\ast \phi^\ast (D + \epsilon H)) + N_\sigma(\psi_i^\ast \phi^\ast (D + \epsilon H)) + \epsilon \psi_i^\ast H_Y = \tau E_i + F. \]
Since $ P_\sigma(\psi_i^\ast \phi^\ast (D + \epsilon H)) + \epsilon \psi_i^\ast H_Y$ is big and in particular effective, we have
\[ \tau E_i + F \geq N_\sigma(\psi_i^\ast \phi^\ast (D + \epsilon H)). \]
Since $E_i$ does not appear on the right-hand side, we have 
\[ \psi_i^\ast \phi^\ast (D + \epsilon H) + \epsilon \psi_i^\ast H_Y - \tau E_i = F \geq  N_\sigma(\psi_i^\ast \phi^\ast (D + \epsilon H)) \] 
implying the claim.

\noindent Again by \cite[III.5.16]{Nak04}, 
\[ P_\sigma(\psi_i^\ast \phi^\ast (D + \epsilon H)) = P_\sigma(\phi_i^\ast \pi_i^\ast (D + \epsilon H)) \leq \phi_i^\ast P_\sigma(\pi_i^\ast (D + \epsilon H)). \] 
But
\[ \phi_i^\ast P_\sigma(\pi_i^\ast (D + \epsilon H)) + \epsilon \psi_i^\ast H_Y \leq \phi_i^\ast B_i + \epsilon \psi_i^\ast H_Y + \frac{1}{i}\phi_i^\ast \pi_i^\ast G \leq A_i + \frac{1}{i}\psi_i^\ast \phi^\ast G. \]
In particular, we conclude that $A_i + \frac{1}{i}\psi_i^\ast \phi^\ast G - \tau E_i$ is pseudoeffective. Therefore,
\begin{eqnarray*}
0 & \leq & (A_i + \frac{1}{i} \psi_i^\ast \phi^\ast G - \tau E_i) \cdot A_i^k \cdot \psi_i^\ast H_Y^{n-k-1} = \\
   &   =  & A_i^{k+1} \cdot \psi_i^\ast H_Y^{n-k-1} + \frac{1}{i} \psi_i^\ast \phi^\ast G \cdot A_i^k \cdot \psi_i^\ast H_Y^{n-k-1} - \tau E_i \cdot A_i^k \cdot \psi_i^\ast H_Y^{n-k-1}.
\end{eqnarray*}
By definition,
\[ 0 \leq A_i^{k+1} \cdot \psi_i^\ast H_Y^{n-k-1} \leq \langle (\phi^\ast (D + \epsilon H) + \epsilon H_Y)^{k+1} \rangle_Y \cdot H_Y^{n-k-1} \]
and
\[ 0 \leq A_i^k \cdot \psi_i^\ast \phi^\ast G \cdot \psi_i^\ast H_Y^{n-k-1} \leq \langle (\phi^\ast (D + \epsilon H) + \epsilon H_Y)^k \rangle_Y \cdot  \phi^\ast G \cdot H_Y^{n-k-1}. \]
So taking the limit over all models $\widetilde{Y}_i$ we obtain
\begin{equation} \label{lim-ineq}
0 \leq \langle (\phi^\ast (D + \epsilon H) + \epsilon H_Y)^{k+1} \rangle_Y \cdot H_Y^{n-k-1} - \tau \langle (\phi^\ast (D + \epsilon H) + \epsilon H_Y)^k \rangle_{Y|E} \cdot H_Y^{n-k-1}.
\end{equation}
If $V = H_1 \cap \ldots \cap H_{n-k-1} \subset Y$ is a $(k+1)$-dimensional complete intersection subvariety cut out by $n-k-1$ very general very ample divisors $H_i \in |H_Y|$, Thm.~\ref{der-resvol-thm} and  Lem.~\ref{res-free-lem} imply that 
\begin{eqnarray*} \label{der-eq}
\frac{d}{dt}_{|t=0} \langle (\phi^\ast (D + \epsilon H) + (\epsilon+t) H_Y)^{k+1} \rangle_{Y|V} & = & 
                                                       (k+1) \cdot \langle (\phi^\ast (D + \epsilon H) + \epsilon H_Y)^k \rangle_{Y|V} \cdot H_Y = \\
         & = & (k+1) \cdot \langle (\phi^\ast (D + \epsilon H) + \epsilon H_Y)^k \rangle_Y \cdot H_Y^{n-k}. 
\end{eqnarray*}
Furthermore, by definition $\lim_{\epsilon \downarrow 0} \langle (\phi^\ast (D + \epsilon H) + \epsilon H_Y)^k \rangle_Y = \langle (\phi^\ast D)^k \rangle_Y $, and the assumption $\nu_{\mathrm{alg}}(D) = k$ implies $\langle (\phi^\ast D)^k \rangle_Y \cdot H_Y^{n-k} > 0$ by Lem.~\ref{neq-prod-pos-int-lem} whereas $\langle (\phi^\ast D)^{k+1} \rangle_Y = 0$. Consequently, there exists $c > 0$ such that $\langle (\phi^\ast (D + \epsilon H) + \epsilon H_Y)^{k+1} \rangle_Y \cdot H^{n-k-1} \leq c \cdot \epsilon$. Then (\ref{lim-ineq}) implies that
\[ \tau \leq \frac{\langle (\phi^\ast (D + \epsilon H) + \epsilon H_Y)^{k+1} \rangle_Y \cdot H_Y^{n-k-1}}{\langle (\phi^\ast (D + \epsilon H) + \epsilon H_Y)^k \rangle_{Y|E} \cdot H_Y^{n-k-1}} \leq \epsilon \cdot \frac{c}{\langle (\phi^\ast (D + \epsilon H) + \epsilon H_Y)^k \rangle_{Y|E} \cdot H_Y^{n-k-1}}. \] 
The denominator of the right-hand side fraction tends to $\langle (\phi^\ast D)^k \rangle_{Y|E} \cdot H_Y^{n-k-1}$ if $\epsilon \rightarrow 0$. By choosing sufficiently general elements $H_1, \ldots, H_{n-k-1} \in |H_Y|$ we may assume that $\phi$ restricted to $E \cap H_1 \ldots \cap H_{n-k-1}$ is a finite morphism onto $W$. If $A_1, \ldots, A_{n-k}$ denote the very ample divisors on $X$ cutting out $W$ there exists $C > 0$ such that
\begin{eqnarray*}
\langle (\phi^\ast D)^k \rangle_{Y|E} \cdot H_Y^{n-k-1} & = & \langle (\phi^\ast D)^k \rangle_{Y|E \cap H_1 \ldots \cap H_{n-k-1}} =
                                                                                                  C \cdot \langle D^k \rangle_{X|W} = \\
   & = & C \cdot \langle D^k \rangle_X \cdot A_1 \cdots A_{n-k}
\end{eqnarray*}
where the first and the last equality follow from Lem.~\ref{res-free-lem} and the middle equality from \cite[Prop.4.20]{Leh13}. By assumption and Lem.~\ref{neq-prod-pos-int-lem} this last product is positive, contradicting the unboundedness of $\frac{\tau}{\epsilon}$ for $\epsilon \rightarrow 0$.
\hfill $\Box$

\section{Birational Invariance and Abundance Conjecture} \label{Bir-Abund-sec}

\noindent To prove that the Abundance Conjecture~\ref{Abund-conj} is equivalent to the MMP-version of the Abundance Conjecture~\ref{MMP-Abund-conj} we need the birational invariance of the numerical dimension of the canonical bundle:
\begin{prop} \label{bir-inv-nd-prop}
Let $X$ be a smooth projective complex variety and $D$ a pseudoeffective divisor on $X$. Let $f: Y \rightarrow X$ be a birational morphism of smooth projective varieties and $E$ an $f$-exceptional effective $\mathbb{R}$-divisor on $Y$. Then:
\[ \nu_X(D) = \nu_Y(f^\ast D + E). \]
\end{prop} 
\begin{proof}
Let $E_1, \ldots, E_k$ be the prime components of $E = \sum_{i=1}^k x_i E_i$, $x_i > 0$.

\noindent Assume first that $D$ is big. Let $\Delta$ be an effective $\mathbb{R}$-divisor $\equiv f^\ast D + E$ on $Y$.

\vspace{0.2cm}

\noindent \textit{Claim.} $\mathrm{mult}_{E_i} \Delta \geq \mathrm{mult}_{E_i}(E) = x_i$ for all $i = 1, \ldots, k$.

\noindent \textit{Proof of Claim.} If $\mathrm{mult}_{E_j} \Delta < \mathrm{mult}_{E_j}(E)$ for a $j \in \{1, \ldots, k\}$ we subtract a multiple of $E_j$ from $\Delta$ and $E$ to obtain $\Delta^\prime \geq 0$, $E^\prime \geq 0$ such that $\Delta^\prime \equiv f^\ast D + E^\prime$ and $0 = \mathrm{mult}_{E_j}\Delta^\prime < \mathrm{mult}_{E_j}E^\prime$. Pushing forward $\Delta^\prime$ we have $f_\ast \Delta^\prime \equiv D$. Hence for real numbers $y_i$ such that $y_j \geq 0$,
\[ f^\ast D \equiv f^\ast f_\ast \Delta^\prime = \Delta^\prime + \sum_{i=1}^k y_i E_i \equiv f^\ast D + \sum_{i=1}^k y_i E_i + E^\prime. \]
Thus, $\sum_{i=1}^k y_i E_i + E^\prime$ is a non-trivial linear combination of the $E_i$ numerically equivalent to $0$.
But this is impossible as numerical classes of $f$-exceptional prime divisors are always linearly independent: On $X$, sufficiently general complete intersection curves $C$ avoid all centers $f(E_i)$ but one, hence the strict transform $\overline{C} \subset Y$ intersects the corresponding prime divisor on $Y$ but none else.
\hfill $\Box$

\vspace{0.2cm}

\noindent The claim implies that $\sigma_{E_i}(f^\ast D + E) \geq \mathrm{mult}_{E_i}(E)$. Taking the limit this also holds when $D$ is only pseudoeffective. Hence $E \leq N_{\sigma}(f^\ast D + E)$, this implies $N_\sigma(f^\ast D + E) - E = N_\sigma(f^\ast D)$, and
\[ P_\sigma(f^\ast D + E) = (f^\ast D + E) - N_\sigma(f^\ast D + E) = f^\ast D - N_\sigma(f^\ast D) = P_\sigma(f^\ast D). \]
The same holds when $\phi: \widetilde{Y} \rightarrow Y$ is a further birational morphism between smooth projective varieties:
\[ P_\sigma(\phi^\ast (f^\ast D + E)) = P_\sigma(\phi^\ast f^\ast D + \phi^\ast E) = P_\sigma(\phi^\ast f^\ast D). \]
Using that the numerical dimension can be defined by $\nu_{\mathrm{Vol, Zar}}$ (see~\ref{nu-vol-zar-ssec} and Def.~\ref{nd-def}) this implies $\nu_Y(f^\ast D + E) = \nu_Y(f^\ast D)$. Defining the numerical dimension via positive intersection products as $\nu_{\mathrm{alg}}$ shows that $\nu_Y(f^\ast D) = \nu_X(D)$, together with the projection formula and the fact that $f^\ast$ defines a homomorphism on the intersection rings.
\end{proof}

\begin{rem} \label{exc-nd-rem}
The proof above also shows that $\nu_Y(f^\ast D) = \nu_Y(f^\ast D + E)$ for a pseudoeffective $\mathbb{Q}$-divisor $D$ on $X$ and an effective $f$-exceptional divisor $E$ on $Y$ even when $X$ is not smooth but only $\mathbb{Q}$-factorial.

\noindent Note that for $\kappa_\sigma$, this invariance was shown in \cite[Prop.2.7(4)\&(7)]{Nak04}, hence Prop.~\ref{bir-inv-nd-prop} already follows from these results and Thm.~\ref{numdim-eq-thm}. However, our proof explicitly uses two different ways of defining the numerical dimension (as Nakayama implicitly does, too), and thus demonstrates much better the usefulness of Thm.~\ref{numdim-eq-thm}.
\end{rem}

\begin{cor} \label{can-bir-nd-cor}
Let $X$ be a non-uniruled smooth projective complex variety and $f: Y \rightarrow X$ a birational morphism between smooth projective varieties. Then:
\[ \nu_X(K_X) = \nu_Y(K_Y). \]
\end{cor}
\begin{proof}
By \cite[Cor.0.3]{BDPP13} the canonical divisors $K_X$ and $K_Y$ are pseudoeffective on the non-uniruled varieties $X$ and $Y$. Hence it is possible to calculate their numerical dimension. Since the pullback of canonical forms through a birational morphism is again a canonical form, there exists an effective $f$-exceptional divisor $E$ such that $K_Y = f^\ast K_X + E$. The corollary follows from Prop.~\ref{bir-inv-nd-prop}.
\end{proof}

\noindent A \textit{minimal model} of a non-uniruled smooth projective complex variety $X$ is a normal variety $S$ such that there exists a sequence of divisorial contractions and flips 
\[ X = X_0 \dasharrow X_1 \dasharrow \cdots \dasharrow X_n = S \]
and $K_S$ is nef (see e.g. \cite[Def.3-3-1 and passim]{Mat02} for further definitions). In particular, $S$ is $\mathbb{Q}$-factorial and has only terminal singularities, that is, every Weil divisor on $S$ is a $\mathbb{Q}$-Cartier divisor and if $f: Y \rightarrow S$ is a birational morphism from a smooth projective variety $Y$ then in the ramification formula
\[ K_Y = f^\ast K_S + \sum a_i E_i, \]
the coefficients $a_i$ of all the $f$-exceptional prime divisors $E_i$ are $> 0$.

\noindent Note that on a minimal model $S$ it is possible to construct intersection products of ($\mathbb{Q}$-)Cartier divisors and to define the numerical triviality of the resulting (rational) cycles (see \cite[19.1]{FIS}). Hence it makes sense to set the numerical dimension of a nef ($\mathbb{Q}$-)Cartier divisor $D$ on $S$ equal to
\[ \nu_S(D) := \max \{ k : D^k \not\equiv 0 \}. \]
If $S$ is smooth this numerical dimension coincides with the one defined in Def.~\ref{nd-def}, by construction of positive intersection products (see \ref{nu-alg-ssec}). 

\noindent The following result of Kawamata \cite{Kaw85b} (see also~\cite{Fuj11} and \cite{Kaw14}) sits at the core of the proof that the two versions of the Abundance Conjecture are equivalent:
\begin{thm}[Kawamata]
On a minimal model $S$, $\kappa_S(K_S) = \nu_S(K_S)$ if and only if $K_S$ is semi-ample.
\end{thm}

\begin{thm} \label{Abund-equiv-thm}
Let $S$ be a minimal model of a non-uniruled smooth projective complex variety $X$. Then
\[ \nu_X(K_X) = \kappa_X(K_X) \Longleftrightarrow K_S\ \mathrm{is\ semi-ample}. \]
\end{thm}
\begin{proof}
By Kawamata's Theorem we only need to prove that $\kappa_X(K_X) = \kappa_S(K_S)$ and $\nu_X(K_X) = \nu_S(K_S)$.

\noindent The first equality follows from using a common resolution

\hspace{4cm}
\xymatrix{
 & Y \ar[dr]^f \ar[dl]_{\phi} & \\
X \ar@{-->}[rr] & & S
}

\noindent of $X$ and $S$ such that $K_Y = f^\ast K_S + E_S = \phi^\ast K_X + E_X$ where $E_S$ and $E_X$ are $f$- resp.\ $\phi$-exceptional effective divisors. Since the sections in $H^0(Y, \mathcal{O}_Y(mK_Y)) = H^0(Y, \mathcal{O}_Y(m\phi^\ast K_X + mE_X))$ can be interpreted both as rational functions on $X$ and $Y$, we have $H^0(Y, \mathcal{O}_Y(mK_Y)) \subset H^0(X, \mathcal{O}_X(mK_X))$, and since $E_X$ is effective the inverse inclusion also holds. Similarly on $S$, and the equality follows.

\noindent For $\nu_X(K_X) = \nu_S(K_S)$ we use Cor.~\ref{can-bir-nd-cor} and Rem.~\ref{exc-nd-rem} to deduce the chain of equalities
\[ \nu_X(K_X) = \nu_Y(K_Y) = \nu_Y(f^\ast K_S + E_S) = \nu_Y(f^\ast K_S) = \nu_S(K_S) \]
where the last equality follows from the projection formula and the fact that $f^\ast$ defines a homomorphism on the intersection rings.
\end{proof}

%\bibliographystyle{alpha}

%\bibliography{/Forschung/Mathematik/doktor}
%\bibliography{doktor}

\newcommand{\etalchar}[1]{$^{#1}$}
\def\cprime{$'$}

\end{document}